\documentclass[12pt, reqno]{article}

\usepackage[usenames]{color}
\usepackage{amssymb}
\usepackage{amsmath}
\usepackage{amsthm}
\usepackage{amsfonts}
\usepackage{amscd}
\usepackage{graphicx}

\usepackage[colorlinks=true,
linkcolor=webgreen,
filecolor=webbrown,
citecolor=webgreen]{hyperref}

\definecolor{webgreen}{rgb}{0,.5,0}
\definecolor{webbrown}{rgb}{.6,0,0}

\usepackage{color}
\usepackage{fullpage}
\usepackage{float}

\usepackage{graphics}
\usepackage{latexsym}
\usepackage{epsf}
\usepackage{breakurl}

\usepackage{mathtools}

\usepackage{subcaption}
\usepackage[capitalize,nameinlink,noabbrev,nosort]{cleveref}

\usepackage[T1,T5]{fontenc} 

\theoremstyle{plain}
\newtheorem{thm}{Theorem}

\newtheorem{lemma}[thm]{Lemma}
\newtheorem{prop}[thm]{Proposition}

\theoremstyle{definition}
\newtheorem{definition}[thm]{Definition}
\newtheorem{example}[thm]{Example}

\theoremstyle{remark}
\newtheorem{rmk}[thm]{Remark}
\newtheorem*{notation}{Notation}

\newcommand{\R}{\mathbb{R}}
\newcommand{\Z}{\mathbb{Z}}
\newcommand{\Zn}{\Z/n\Z}
\newcommand{\Zx}{(\Zn)^\times}
\DeclareMathOperator{\Aut}{Aut}
\DeclareMathOperator{\Conj}{Conj}
\DeclareMathOperator{\id}{id}
\DeclareMathOperator{\PSL}{PSL}
\DeclareMathOperator{\PGL}{PGL}
\newcommand{\perm}{_{\mathrm{perm}}}
\newcommand{\surj}{\twoheadrightarrow}
\newcommand{\inv}{^{-1}}
\renewcommand{\phi}{\varphi}

\newcommand{\seqnum}[1]{\href{https://oeis.org/#1}{\underline{#1}}}

\begin{document}

\begin{center}
\vskip 1cm{\LARGE\bf 
Enumeration of Virtual Quandles\\
\vskip .1in 
up to Isomorphism
}
\vskip 1cm
\large
L\d\uhorn c Ta\\
Department of Mathematics\\
University of Pittsburgh\\
Pittsburgh, PA 15260\\
USA\\
\href{mailto:ldt37@pitt.edu}{\tt ldt37@pitt.edu}
\end{center}

\vskip .2 in

\begin{abstract}
Virtual racks and virtual quandles are nonassociative algebraic structures based on the Reidemeister moves of virtual knots. In this note, we enumerate virtual dihedral quandles and several families of virtual permutation racks and virtual conjugation quandles up to isomorphism. We also classify virtual racks and virtual quandles up to order $8$ using a computer search. These classifications are based on the conjugacy class structures of rack automorphism groups. In particular, we compute class numbers of holomorphs of finite cyclic groups, which may be of independent interest.
\end{abstract}

\section{Introduction}

In this note, we enumerate virtual dihedral quandles and other families of virtual racks up to isomorphism. These classifications are based on the conjugacy class structures of the automorphism groups of the underlying racks. Using a computer search, we also classify virtual racks and virtual quandles up to order $8$. We provide our code and data in a GitHub repository~\cite{code}. 

\subsection{Motivations}

In 1992, Fenn and Rourke~\cite{fenn} introduced nonassociative algebraic structures called \emph{racks} to construct complete invariants of irreducible framed links in connected $3$-manifolds. Racks generalize algebraic structures called \emph{quandles}, which Joyce~\cite{joyce} and Matveev~\cite{matveev} independently introduced in 1982 to construct complete invariants of links in $\R^3$ and $S^3$, and \emph{kei}, which Takasaki~\cite{takasaki} introduced in 1943 to study Riemannian symmetric spaces. Racks, quandles, and kei have enjoyed significant study as knot invariants in geometric topology and in their own rights in quantum algebra and group theory. We refer the reader to a survey of Nelson~\cite{intro} and the book of Elhamdadi and Nelson~\cite{quandlebook} for accessible introductions to quandle theory, the book of Nosaka~\cite{book} for a reference on racks as they concern low-dimensional topology, and an article of Elhamdadi~\cite{survey} for a survey of the algebraic state of the art.

In 2002, Manturov~\cite{manturov} introduced \emph{virtual racks}, which are racks equipped with a distinguished automorphism. Virtual racks generalize \emph{virtual quandles}, which define invariants of virtual links in thickened surfaces~\cite{biquandle,manturov}, links in certain lens spaces~\cite{lens}, and long virtual links~\cite{long}. Virtual racks also generalize \emph{GL-racks}, which can distinguish Legendrian links in $\R^3$ or $S^3$ with the standard contact structure~\cite{karmakar,bi,ta}, including Legendrian knots not distinguishable by their classical or homological invariants~\cite{ta2}. More abstractly, virtual racks model braided objects in symmetric categories and actions of virtual braid groups~\cite{lebed}. Recent papers have further generalized virtual racks to construct invariants of multi-virtual links, which generalize virtual links~\cite{multi,multi2}.

Finite virtual quandles are particularly important because they define counting invariants and cohomological invariants of virtual links~\cite{cocycle,biquandle,manturov}. In particular, various authors have used \emph{virtual dihedral quandles} to study the Kauffman--Harary conjecture and Fox $n$-coloring invariants of virtual links~\cite{fox2,farinati,fox}. The classification of finite virtual quandles is therefore a natural question that extends the classification problems for finite quandles~\cite{finite,vy} and finite GL-racks~\cite[Sec.\ 3]{karmakar}.

\subsection{Structure of the paper}

In Section~\ref{sec:prelims}, we give definitions and examples of virtual racks and virtual quandles.

In Section~\ref{sec:obs}, we record a key observation relating virtual rack isomorphism classes to conjugacy classes of rack automorphism groups. As an application, we classify virtual racks and virtual quandles up to order $8$. As another application, we introduce an integer-valued rack invariant $v(R)$. We compute $v(R)$ for certain permutation racks and conjugation quandles.

In Section~\ref{sec:main}, we classify virtual dihedral quandles up to isomorphism. In particular, we obtain the following enumeration.

\begin{thm}\label{thm:main}
    Let $n\in \Z^+$ be a positive integer, let $R_n$ be the dihedral quandle of order $n$, and let $\phi$ denote Euler's totient function. Then the number of isomorphism classes of virtual quandles whose underlying quandles are isomorphic to $R_n$ is
    \begin{equation}\label{eq:seq}
        v(R_n)=\phi(n)\sum_{d\mid n}\frac{1}{\phi(d)}.
    \end{equation} 
\end{thm}

\begin{rmk}
    The integer sequence $(v(R_n))_{n\geq 1}$ is OEIS sequence \seqnum{A069208}~\cite{OEIS}. Our proof of Theorem~\ref{thm:main} shows that $(v(R_n))_{n\geq 1}$ is also the sequence of class numbers of the holomorphs ${\Zn\rtimes\Zx}$ of finite cyclic groups; see Theorem \ref{thm:main2}. 
\end{rmk}
    
\begin{notation}
    We use the following notation throughout this paper. 
    
    We denote the composition of functions $\chi:X\to Y$ and $\psi:Y\to Z$ by $\psi\chi$.
    
    Given a set $X$, let $S_X$ be the symmetric group of $X$. 
    
    Given a group $G$, let $Z(G)$ denote the center of $G$. Given an element $g\in G$, let $C_{S_X}(g)$ denote the centralizer of $g$ in $G$.

    For all positive integers $n\in\Z^+$, let $\Zn$ denote the cyclic group of order $n$, and let $\Zx$ denote its multiplicative group of units.
\end{notation}

\section{Preliminaries}\label{sec:prelims}
\subsection{Racks}

	\begin{definition}
		Let $X$ be a set, let $s:X\to S_X$ be a map, and write $s_x:=s(x)$ for all elements $x\in X$. We call the pair $R:=(X,s)$ a \emph{rack} if \begin{equation*}
		s_xs_y=s_{s_x(y)}s_x
	\end{equation*}
        for all $x,y\in X$.  
        If in addition $s_x(x)=x$ for all $x\in X$, then we say that $R$ is a \emph{quandle}.
		The \emph{order} of $R$ is the cardinality of $X$.
	\end{definition}

\begin{rmk}
    Some authors define racks and quandles as sets $X$ with a right-distributive nonassociative binary operation $\triangleright:X\times X\to X$ satisfying certain axioms. These two definitions are equivalent \cite{lens,survey} via the formula
    \[
    s_y(x)=x\triangleright y.
    \]
\end{rmk}

	\begin{example}\cite[Ex.\ 99]{quandlebook}
		Let $X$ be a set, and fix a permutation $\sigma\in S_X$. Define $s:X\rightarrow S_X$ by $x\mapsto \sigma$, so that $s_x(y)=\sigma(y)$ for all $x,y\in X$. Then $(X,\sigma)\perm:=(X,s)$ is a rack called a \emph{permutation rack}, \emph{cyclic rack}, or \emph{constant action rack}.
        
        Note that $(X,\sigma)\perm$ is a quandle if and only if $\sigma$ is the identity map $\id_X$. In this case, we call $(X,\id_X)\perm$ a \emph{trivial quandle}.
	\end{example}

    	\begin{example}\cite[Ex.\ 2.13]{book}
	Let $X$ be a union of conjugacy classes in a group $G$, and define $c:X\rightarrow S_X$ by sending any element $x\in X$ to the conjugation map \[c_x:=[y\mapsto xyx\inv].\]
		Then $\Conj X:=(X,c)$ is a quandle called a \emph{conjugation quandle} or \emph{conjugacy quandle}. 
        Note that $\Conj G$ is a trivial quandle if and only if $G$ is abelian.
\end{example}

        \begin{example}\cite[Ex.\ 54]{quandlebook}
		Let $G$ be an abelian additive group. Define $s:G\to S_G$ by setting \[s_g(h):=2h-g\] for all elements $g,h\in G$. Then $T(G):=(G,s)$ is a quandle called a \emph{Takasaki kei}. 
        
        In particular, if $n\in\Z^+$ and $G=\Zn$, then we call $T(\Zn)$ the \emph{dihedral quandle} of order $n$. We denote $T(\Zn)$ by $R_n$.
	\end{example}

\subsection{Virtual racks}
Virtual racks are racks equipped with a rack automorphism.
    
    	\begin{definition}
		Given two racks $(X,s)$ and $(Y,t)$, we say that a map $\psi:X\to Y$ is a \emph{rack homomorphism} if \[\psi s_x=t_{\psi(x)}\psi\] for all $x\in X$. A \emph{rack isomorphism} is a bijective rack homomorphism. \emph{Automorphisms} of racks are isomorphisms from a rack to itself. 
        We denote the \emph{automorphism group} of a rack $R$ by $\Aut R$. 
	\end{definition}

    \begin{definition}
    A \emph{virtual rack} is a pair $(R,f)$ where $R$ is a rack and $f\in\Aut R$ is a rack automorphism. We say that $f$ is a \emph{virtual structure} on $R$.  
    If in addition $R$ is a quandle, then we call $(R,f)$ a \emph{virtual quandle}. We say that $R$ is the \emph{underlying rack} (resp.\ \emph{underlying quandle}) of $(R,f)$.
\end{definition}

\begin{example}\label{ex:perm} 
    Let $P=(X,\sigma)\perm$ be a permutation rack. Evidently, $\Aut P=C_{S_X}(\sigma)$, so a virtual structure on $P$ is precisely a permutation of $X$ that commutes with $\sigma$.
\end{example}

\begin{example}
    Let $G$ be a group. We consider several virtual structures on $\Conj G$. Fix an element $g\in G$. 
    \begin{itemize}
        \item Consider the left-multiplication map $m_g:=[h\mapsto gh]$. Then $(\Conj G,m_g)$ is a virtual quandle if and only if $g\in Z(G)$~\cite[Ex.\ 3.6]{bi}.
        \item Consider the conjugation map $c_g:=[h\mapsto ghg\inv]$. Then $(\Conj G,c_g)$ is a virtual quandle~\cite[Lemma 1]{manturov}. More generally, given any rack $R=(X,s)$ and any element $x\in X$, the pair $(R,s_x)$ is a virtual rack.
    \end{itemize}
\end{example}

\begin{example}\label{ex:dihedral}
    Let $n\in\Z^+$ be a positive integer, and let $R_n$ be the dihedral quandle of order $n$. A result of Elhamdadi, Macquarrie, and Restrepo~\cite[Thm.\ 2.3]{aut} states that $\Aut R_n$ is the group of affine transformations of $\Zn$, which is isomorphic to the holomorph ${\Zn\rtimes\Zx}$ of $\Zn$. Hence, virtual structures on $R_n$ are precisely affine transformations
    \[
    T_{a,u}:=[x\mapsto a+ux]
    \]
    with $a\in \Zn$ and $u\in \Zx$.
\end{example}

\begin{definition}\label{def:auto}
    Let $V_1=(R_1,f_1)$ and $V_2=(R_2,f_2)$ be virtual racks. A \emph{virtual rack homomorphism} from $V_1$ to $V_2$ is a rack homomorphism $\psi$ from $R_1$ to $R_2$ that intertwines the virtual structures: 
        \[\psi f_1=f_2\psi.\] 
    \emph{Virtual rack isomorphisms} are bijective virtual rack homomorphisms.
\end{definition}

\section{Enumeration results}\label{sec:obs}

\subsection{Key observation}

The following is immediately evident from Definition \ref{def:auto}.

\begin{prop}\label{prop:key}
    Let $f_1,f_2$ be two virtual structures on a rack $R$. Then $(R,f_1)\cong(R,f_2)$ if and only if $f_1$ and $f_2$ are conjugate in $\Aut R$.
\end{prop}

In other words, classifying virtual structures on a given rack up to isomorphism is equivalent to characterizing the conjugacy classes of the rack's automorphism group. 

\subsection{Applications of Proposition \ref{prop:key}}

\subsubsection{A library of virtual racks}

Using the computer algebra system \texttt{GAP} \cite{GAP4}, we implemented an algorithm that exhaustively searches for conjugacy relations in automorphism groups of finite racks. By Proposition \ref{prop:key}, this search classifies finite virtual racks and virtual quandles up to isomorphism. Our implementation uses Vojt\v echovsk\'y and Yang's library of racks and quandles~\cite{vy} in a way similar to the author's previous classification of GL-racks up to order $8$~\cite[Appendix A]{ta2}. 

We were able to complete this search for virtual racks and virtual quandles up to order~$8$. We provide our \texttt{GAP} code and tabulate our data in a GitHub repository \cite{code}. We also enumerate our data in OEIS sequences \seqnum{A385040} and \seqnum{A385041}~\cite{OEIS}.

\subsubsection{Class numbers as rack invariants}

\begin{notation}
We use the following notation for the remainder of this paper.

For all racks $R$, let $v(R)$ denote the number of isomorphism classes of virtual racks whose underlying rack is isomorphic to $R$. 

    Given a group $G$, let $k(G)$ be the class number of $G$, that is, the number of conjugacy classes of $G$. 
\end{notation}

Proposition \ref{prop:key} implies that for all racks $R$,
\begin{equation}\label{eq:count}
    v(R)=k(\Aut R).
\end{equation}
In particular, $v(R)$ is a rack invariant. 
The rest of this note is dedicated to using Equation~\eqref{eq:count} to compute $v(R)$ for various families of racks.

\begin{notation}
    For all nonnegative integers $n\geq 0$, let $p(n)$ denote the number of partitions of~$n$~\cite[\seqnum{A000041}]{OEIS}.
\end{notation}

\begin{example}\label{ex:perm2}
    Let $P=(X,\sigma)\perm$ be a permutation rack. By Example~\ref{ex:perm} and Equation~\eqref{eq:count}, \[v(P)=k(\Aut P)=k(C_{S_X}(\sigma)).\] 
    
    For example, let $X$ be a finite set of cardinality $n\in\Z^+$. If $P$ is a trivial quandle, then $\Aut P=S_X$, so $v(P)=p(n)$. If, instead, $\sigma\in S_n$ is an $n$-cycle, then $C_{S_X}(\sigma)=\langle\sigma\rangle \cong \Zn$, so $v(P)=n$.
\end{example}

\begin{rmk}
    Example~\ref{ex:perm2} coincides with the classification of GL-racks whose underlying racks are permutation racks~\cite[Subsec.\ 4.2]{ta2}.
\end{rmk}

\subsubsection{Virtual conjugation quandles}

The following result of Bardakov, Nasybullov, and Singh~\cite[Cor.\ 2]{centerless} gives a simple way to compute $v(\Conj G)$ when $G$ is a centerless group.

\begin{lemma}\label{lemma:centerless}
    Let $G$ be a group, and let $\Aut_{\mathsf{Grp}}(G)$ denote its automorphism group. Then $Z(G)=1$ if and only if \[\Aut_{\mathsf{Grp}}(G)=\Aut (\Conj G).\] 
\end{lemma}

\begin{example}\label{ex:sn}
    Let $n\in\Z^+$ be a positive integer, and let $S_n$ denote the symmetric group on $n$ letters. We compute $v(\Conj S_n)$ as follows. Since $\Conj S_2$ is a trivial quandle, Example~\ref{ex:perm2} yields $v(\Conj S_2)=p(2)=2$. 
    If $n\neq 2$, then $Z(S_n)=1$, so Lemma~\ref{lemma:centerless} implies that
    \[
    \Aut (\Conj S_n)=\begin{cases}
        S_n, &\text{if $n\neq 6$;}\\
        S_6\rtimes \Z/2\Z,&\text{if $n=6$.}
    \end{cases}
    \]
    Since $k(S_n)=p(n)$ for all $n\in\Z^+$ and $k(S_6\rtimes\Z/2\Z)=13$, Equation~\eqref{eq:count} and the above discussion yield \[v(\Conj S_n)=\begin{cases}
        p(n),&\text{if $n\neq 6$;}\\
        13,&\text{if $n=6$.}
    \end{cases}\]
\end{example}

\begin{example}
    Let $n\in\Z^+$ be a positive integer, and let $A_n$ denote the alternating group on $n$ letters. An argument similar to that of Example~\ref{ex:sn} yields
    \[
    v(\Conj A_n)=\begin{cases}
        p(n),&\text{if $n\neq 2,6$;}\\
        1, & \text{if $n=2$,}\\
        13,&\text{if $n=6$.}
        \end{cases}
    \]
\end{example}

\begin{example}
    For all odd primes $p\geq 3$, the projective special linear group $\PSL(2,p)$ is centerless, and its automorphism group is isomorphic to the projective linear group $\PGL(2,p)$. Hence, Lemma~\ref{lemma:centerless} and Equation~\eqref{eq:count} yield
    \[
    v(\Conj \PSL(2,p)) =k(\PGL(2,p))=(p-1)(p+1).
    \]
    See OEIS sequence \seqnum{A084920}~\cite{OEIS}.
\end{example}

\section{Virtual dihedral quandles}\label{sec:main}
In this section, we classify virtual dihedral quandles up to isomorphism. In particular, we prove Theorem~\ref{thm:main}.

\begin{notation}    
    For the remainder of the paper, let $n\in\Z^+$ be a positive integer, let $A:=\Z/n\Z$, and let $G:=A\rtimes A^\times$ be the holomorph of $A$. 
\end{notation}

\subsection{Classification}

First, we recall Example~\ref{ex:dihedral}: Virtual structures on the dihedral quandle $R_n$ of order $n$ are precisely the affine transformations $T_{a,u}:A\to A$ defined by
\[
x\mapsto a+ux
\]
with $a\in A$ and $u\in A^\times$. Indeed, the mapping $T_{a,u}\mapsto (a,u)$ defines a group isomorphism from $\Aut R_n$ to $G$~\cite[Thm.\ 2.3]{aut}.

In the following, we classify virtual dihedral quandles up to isomorphism.

\begin{prop}\label{prop:dihedral}
    Let $n\in\Z^+$ be a positive integer, and let $T_{a,u},T_{b,v}$ be virtual structures on the dihedral quandle $R_n$ of order $n$. Then $(R_n,T_{a,u})\cong (R_n,T_{b,v})$ if and only if $u=v$ and
    \[
    b=ya+(1-u)x
    \]
    for some element $(x,y)\in G$.
\end{prop}

By Proposition~\ref{prop:key}, Proposition~\ref{prop:dihedral} is equivalent to the following.

\begin{prop}\label{prop:holomorph}
For all $u\in A^\times$, define the group 
\begin{equation}\label{eq:group}
    A_u:=(1-u)A\leq A.
\end{equation}
    Then two elements $(a,u),(b,v)\in G$ are conjugate if and only if \textup{(}i\textup{)} $u=v$ and \textup{(}ii\textup{)} $a$ and $b$ lie in the same orbit under the action of $A^\times$ on $A/A_u$.
\end{prop}

\begin{proof}
    Conjugation in $G$ is given by 
    \[
    (x,y)(a,u)(x,y)\inv=(x+ya,yu)(-y\inv x,y\inv)=(ya+(1-u)x,u).
    \]
    The claim follows immediately from this computation.
\end{proof}

\subsection{Enumeration}

Finally, we prove Theorem~\ref{thm:main}. By Equation~\eqref{eq:count}, Theorem~\ref{thm:main} is equivalent to the following.

\begin{thm}\label{thm:main2}
    Let $\phi$ denote Euler's totient function. 
    If $A=\Z/n\Z$ and ${G=A\rtimes A^\times}$, then
    \[
    k(G)=
    \phi(n)\sum_{d\mid n}\frac{1}{\phi(d)}.
    \]
\end{thm}

\begin{proof}
    For all elements $u\in A^\times$, define the group $A_u$ as in Equation~\eqref{eq:group}, and let $\psi(u)$ denote the number of orbits of the action of $A^\times$ on $A/A_u$. 
    By Proposition~\ref{prop:holomorph},
    \[
    k(G)=\sum_{u\in A^\times}\psi(u).
    \]
We compute the right-hand side. 
For each element $u\in A^\times$, consider the integer \[
    d_u:=\gcd(n,1-u).
    \]
    By construction,
    \begin{equation}\label{eq:cyclic}
    A/A_u\cong\Z/d_u\Z.
    \end{equation}
    Let $\tau(d_u)$ be the number of divisors of $d_u$. 
    By Equation \eqref{eq:cyclic} and a standard result of group theory~\cite[Lemma 1.3]{kohl}, the number of orbits of the action of $(A/A_u)^\times$ on $A/A_u$ equals $\tau(d_u)$. Since $A^\times$ projects onto $(A/A_u)^\times$, it follows that $\psi(u)=\tau(d_u)$. Therefore,
\begin{equation}\label{eq:ker}
    k(G)=\sum_{u\in A^\times}\tau(d_u)=\sum_{u\in A^\times} \sum_{d\mid d_u}1= \sum_{d\mid n} \sum_{\substack{u\in A^\times\\
    d\mid d_u}}1,
\end{equation}
where in the last equality we have used the fact that the sums are finite. 

Note that if $d\mid n$ and $u\in A^\times$, then $d\mid d_u$ if and only if $u\equiv 1$ (mod $d$); the latter condition holds if and only if $u$ lies in the kernel of the projection $\pi_d:A^\times\surj(\Z/d\Z)^\times$. 
In other words, the rightmost sum in Equation~\eqref{eq:ker} equals the cardinality of $\ker\pi_d$. 
By the first isomorphism theorem, this cardinality equals
\[
|\ker\pi_d|=\frac{|A^\times|}{|(\Z/d\Z)^\times|}=\frac{\phi(n)}{\phi(d)}.
\]
Hence,
\[
k(G)=\sum_{d\mid n} |\ker\pi_d|=\sum_{d\mid n}\frac{\varphi(n)}{\varphi(d)}.
\]
\end{proof}

\section{Acknowledgments}
I first learned about quandles and GL-racks at the UnKnot V conference in 2024; I thank Jose Ceniceros and Peyton Wood for introducing me to the theory. I also thank Sam Raskin for advising a previous paper that helped inspire this note.

\bigskip
\hrule
\bigskip

\noindent 2020 {\it Mathematics Subject Classification}:
Primary 20N02; Secondary 20B25, 20E45, 57K12.

\noindent \emph{Keywords}: Classification, conjugacy class, dihedral quandle, enumeration, rack, quandle, virtual quandle.

\bigskip
\hrule
\bigskip

\noindent (Concerned with sequences
\seqnum{A000041},
\seqnum{A069208},
\seqnum{A084920},
\seqnum{A385040}, and
\seqnum{A385041}.)

\bigskip
\hrule
\bigskip

\vspace*{+.1in}
\noindent
Received June 17 2025.

\bigskip
\hrule
\bigskip

\noindent
Return to
\href{https://cs.uwaterloo.ca/journals/JIS/}{Journal of Integer Sequences home page}.
\vskip .1in

\end{document}